\documentclass[11pt]{article}

\usepackage[utf8]{inputenc}
\usepackage{amsmath,amsthm,amssymb,enumerate,framed,mdwlist}
\usepackage{color,colortab}
\usepackage[square,numbers,sort&compress]{natbib}
\usepackage{bm}
\usepackage{graphicx}
\usepackage{relsize}

\newcommand{\N}{\mathbb{N}}

\newcommand{\R}{\mathbb{R}}
\newcommand{\C}{\mathbb{C}}
\newcommand{\E}{\mathbb{E}}
\newcommand{\conv}{\operatorname{conv}}

\newcommand{\Tr}{\operatorname{Tr}}

\newcommand{\cR}{\mathcal{R}}
\newcommand{\cP}{\mathcal{P}}
\newcommand{\cL}{\mathcal{L}}

\newcommand{\D}{\mathcal{D}}
\newcommand{\cN}{\mathcal{N}}

\newcommand{\cA}{\mathcal{A}}

\newcommand{\Proj}{\operatorname{Proj}}
\newcommand{\argmax}{\operatorname{argmax}}
\newcommand\norm[1]{\left\lVert#1\right\rVert}

\renewcommand{\epsilon}{\varepsilon}

\usepackage{bigints}

\newtheoremstyle{mythmstyle}
	{\topsep}
	{\topsep}
	{\itshape}
	{}
	{\scshape}
	{.}
	{3pt}
	{}
\theoremstyle{mythmstyle}

\newtheorem{nn}{}[section]
\newtheorem{lemma}[nn]{Lemma}
\newtheorem{theorem}[nn]{Theorem}

\newtheorem{prop}[nn]{Proposition}
\newtheorem{definition}[nn]{Definition}
\newtheorem{claim}[nn]{Claim}

\newtheorem{REMARK}[nn]{Remark}

\newenvironment{cpf}{\begin{trivlist} \item[] {\em Proof of Claim.}}{\hspace*{\stretch{1}} $\diamond$ \end{trivlist}}

\def\ve#1{\mathchoice{\mbox{\boldmath$\displaystyle\bf#1$}}
{\mbox{\boldmath$\textstyle\bf#1$}}
{\mbox{\boldmath$\scriptstyle\bf#1$}}
{\mbox{\boldmath$\scriptscriptstyle\bf#1$}}}

\newcommand{\x}{{\ve x}}

\numberwithin{equation}{section}

\allowdisplaybreaks

\begin{document}

\title{Admissibility of solution estimators for stochastic optimization}

\author{Amitabh Basu\footnote{Department of Applied Mathematics and Statistics, The Johns Hopkins University. Amitabh Basu and Tu Nguyen were supported by NSF grant CMMI1452820 and ONR grant N000141812096. Amitabh Basu also acknowledges support from NSF grant CCF1934979.} \and Tu Nguyen\footnotemark[1] \and Ao Sun\footnotemark[1]}


\maketitle

\begin{abstract} 
We look at stochastic optimization problems through the lens of statistical decision theory. In particular, we address admissibility, in the statistical decision theory sense, of the natural sample average estimator for a stochastic optimization problem (which is also known as the {\em empirical risk minimization (ERM) rule} in learning literature). It is well known that for some simple stochastic optimization problems, the sample average estimator may not be admissible. This is known as {\em Stein's paradox} in the statistics literature. We show in this paper that for optimizing stochastic linear functions over compact sets, the sample average estimator {\em is} admissible. Moreover, we study problems with convex quadratic objectives subject to box constraints. Stein's paradox holds when there are no constraints and the dimension of the problem is at least three. We show that in the presence of box constraints, admissibility is recovered for dimensions $3$ and $4$. 
\end{abstract}

\section{Introduction}
\label{Intro}

A large class of stochastic optimization problems can be formulated in the following way:

\begin{equation}\label{eq:original-so}\min_{x \in X}\{f(x) := \E_{\xi}[F(x, \xi)]\},\end{equation} where $X\subseteq \R^d$ is a fixed feasible region, $\xi$ is a random variable taking values in $\R^m$, and $F:\R^d \times \R^m \to \R$. We wish to solve this problem with access to independent samples of $\xi$. 
The following are two classical examples:

\begin{enumerate}
\item Consider a learning problem with access to labeled samples $(z,y) \in \R^n \times \R$ from some distribution and the goal is to find a function $f \in \mathcal{F}$ in a finitely parametrized hypothesis class $\mathcal{F}$ (e.g., all neural network functions with a fixed architecture) that minimizes expected loss, where the loss function is given by $\ell: \R \times \R \to \R_+$. One can model this using~\eqref{eq:original-so} by setting $d$ to be the number of parameters for $\mathcal{F}$, $m = n+1$, $X\subseteq \R^d$ is the subset that describes $\mathcal{F}$ via the parameters, and $F(f,(z,y)) = \ell(f(z),y)$.
\item When $d=m$, $F(x, \xi) = \| x - \xi\|^2$, $X = \R^d$ and $\xi$ is distributed with mean $\mu$, \eqref{eq:original-so} becomes $$\min_{x \in \R^d}\;\E_{\xi}[\|x - \xi\|^2] = \min_{x \in \R^d}\;\|x - \E[\xi]\|^2 + \mathbb{V}[\xi]$$
In particular, if one knows $\mu:= \E[\xi]$, the optimal solution is given by $x = \mu$. Thus, this stochastic optimization problem becomes equivalent to the classical statistics problem of estimating the mean of the distribution of $\xi$, given access to independent samples. 
\end{enumerate}

We would like to emphasize our data-driven viewpoint on the problem~\eqref{eq:original-so}. In particular, we will not assume detailed knowledge of the distribution of the random variable $\xi$, but only assume that it comes from a large family of distributions. More specifically, we will not assume knowledge of means or higher order moments, and certainly not the exact distribution of $\xi$. This is in contrast to some approaches within the stochastic optimization literature that proceed on the assumption that such detailed knowledge of the distribution is at hand. Such an approach would rewrite~\eqref{eq:original-so} by finding an analytic expression for the expectation $\E_{\xi}[F(x, \xi)]$ (in terms of the known parameters of the distribution of $\xi$), perhaps with some guaranteed approximation if an exact analysis is difficult. The problem then becomes a {\em deterministic} optimization problem\footnote{Note that we are {\em not} referring to what stochastic optimization literature refers to as the deterministic model where all appearances of any random variable in the problem are replaced by its expectation. We are really talking about what the stochastic optimization community would refer to as the stochastic problem and a stochastic solution.}, often a very complicated and difficult one, which is then attacked using novel and innovative ideas of mathematical optimization. See~\cite{birge2011introduction} for a textbook exposition of this viewpoint.

In contrast, as mentioned above, we will assume that the true distribution of $\xi$ comes from a postulated large family of (structured) distributions, and we assume that we have access to data points $\xi^1, \xi^2, \ldots$ drawn independently from the true distribution of $\xi$. This makes our approach distinctly statistical and data-driven in nature. We ``learn" or glean information about the distribution of $\xi$ from the data, which we then use to ``solve"~\eqref{eq:original-so}. 
Statistical decision theory becomes a natural framework for such a viewpoint, to formalize what it even means to ``solve" the problem after ``learning" about the distribution from data. 
We briefly review relevant terminology from statistical decision theory below.

We do not mean to imply that a statistical perspective on stochastic optimization has never been studied before this paper. This is far from true; see~\cite[Chapter 9]{birge2011introduction} and~\cite[Chapter 5]{shapiro2009lectures} for detailed discussions of statistical approaches and methods in stochastic optimization. In~\cite{liyanage2005practical} and~\cite{chu2008solving}, the authors introduce a statistical decision theory perspective that is essentially the same as our framework. In recent parlance, ``data-driven optimization" has been used to describe the statistical viewpoint and has a vast literature; some recent papers closely related to our work are~\cite{gupta2017small,elmachtoub2017smart,bertsimas2018data,esfahani2018data,van2017data}, with~\cite{gupta2017small,elmachtoub2017smart} particularly close in spirit to this paper.  Nevertheless, our perspective is quite different and follows in the footsteps of the inspirational paper of Davarnia and Cornuejols~\cite{davarnia2017estimation}.

\subsection{Statistical decision theory and admissibility}

Statistical decision theory is a mathematical framework for modeling decision making in the face of uncertain or incomplete information. One models the uncertainty by a set of {\em states of nature} denoted by $\Theta$. The decision making process is to choose an {\em action} from a set $\mathcal{A}$ that performs best in a given state of nature $\theta$. To take our stochastic optimization setting, the set of states of nature is given by the family $\D$ of distributions that we believe the true distribution of $\xi$ comes from, and the set of actions $\cA$ is the feasible region $X$, i.e., select $x \in X$ that minimizes $f(x) := \E_{\xi \sim D}[F(x, \xi)]$. In the general framework of decision theory, one defines a {\em loss function} $\cL : \Theta \times \cA \to \R_+$ to evaluate the performance of an action $a\in \cA$ against a state of nature $\theta\in \Theta$. The smaller $\cL(\theta, a)$ is, the better $a \in \cA$ does with respect to the state $\theta\in \Theta$\footnote{We caution the reader that the use of the words ``loss" and ``risk" in statistical decision theory are somewhat different from their use in machine learning literature. In machine learning, the function $F(x, \xi)$ is usually referred to as ``loss" and the function $f(x)$ is referred to as ``risk" in~\eqref{eq:original-so}. Thus Example 1. above becomes a ``risk minimization" problem with an associated ``empirical risk minimization (ERM)" problem when one replaces the expectation by a sample average.}. In our setting of stochastic optimization, we take an action $\hat x \in X$. The natural way to evaluate its performance is via the so-called {\em optimality gap}, i.e., how close is $f(\hat x)$ to the optimal value of~\eqref{eq:original-so}. Therefore, the following is a natural loss function for stochastic optimization: \begin{align}\label{eq:loss-so}\cL(D, x) & := f(x) - f(x(D)) &  \nonumber \\
&=  \E_{\xi \sim D}[F(x, \xi)] - \E_{\xi \sim D}[F(x(D), \xi)],\end{align} where $x(D)$ is an optimal solution to~\eqref{eq:original-so} when $\xi \sim D$. 

The {\em statistical} aspect of statistical decision theory comes from the fact that the state $\theta$ is not revealed directly, but only through data/observations based on $\theta$ that can be noisy or incomplete. 
This is formalized by postulating a parameterized family of probability distributions $\cP:= \{P_\theta: \theta \in \Theta\}$ on a common {\em sample space} $\mathcal{X}$. After observing a realization $y \in \mathcal{X}$ of this random variable, one forms an opinion about what the possible state is and one chooses an action $a\in \cA$. Formally, a {\em decision rule} is a function $\delta: \mathcal{X} \to \mathcal{A}$ giving an action $\delta(y) \in \cA$ when data $y \in \mathcal{X}$ is observed. To take our particular setting of stochastic optimization, one observes data points $\xi^1, \xi^2, \ldots, \xi^n$ that are i.i.d. realizations of $\xi \sim D$; thus, $\mathcal{X} = \underbrace{\R^m \times \R^m \times \ldots \times \R^m}_{n \textrm{ times}}$ with distributions $\cP:= \{ \underbrace{D \times D \times \ldots \times D}_{n \textrm{ times}}: D \in \D \}$ on $\mathcal{X}$ parameterized by the states $D \in \D$. 

Finally, one evaluates decision rules by averaging over the data, defining the {\em risk function} $$\cR(\theta, \delta) := \E_{y \sim P_\theta}[\cL(\theta,\delta(y))].$$

One can think of the risk function as mapping a decision rule to a nonnegative function on the class of distributions $\cP$; this function is sometimes called {\em the risk of the decision rule}. A decision rule is ``good" if its risk has ``low" values. A basic criterion for choosing decision rules in then the following. 
We say that $\delta'$ {\em weakly dominates} $\delta$ if $\cR(\theta, \delta') \leq \cR(\theta,\delta)$ for all $\theta\in \Theta$. We say that $\delta'$ {\em dominates} $\delta$ if, in addition, $\cR(\hat\theta, \delta') < \cR(\hat \theta,\delta)$ for some $\hat \theta\in \Theta$. A decision rule $\delta$ is said to be {\em inadmissible} if there exists another decision rule $\delta'$ that dominates $\delta$. A decision rule $\delta$ is said to be {\em admissible} if it is not dominated by any other decision rule. In-depth discussions of general statistical decision theory can be found in~\cite{berger2013statistical,bickel2015mathematical,lehmann2006theory}.

%
%
%
\subsection{Admissibility of the sample average estimator and our results}\label{sec:results}

We would like to study the admissibility of natural decision rules for solving~\eqref{eq:original-so}. 
As explained above, we put this in the decision theoretical framework by setting the sample space $\mathcal{X} = \underbrace{\R^m \times \ldots \times \R^m}_{n \textrm{ times}}$, where $n$ is the number of i.i.d. observations one makes for $\xi \sim D$ for $D \in \D$, and $\D$ is a fixed family of distributions. A decision rule is now a map $\delta: \underbrace{\R^m \times \R^m \times \ldots \times \R^m}_{n \textrm{ times}} \to X$. The class of distributions on $\mathcal{X}$ is $\cP = \{ \underbrace{D \times D \times \ldots \times D}_{n \textrm{ times}}: D \in \D \}$. The loss function is defined as in~\eqref{eq:loss-so}. 
In this paper, we wish to study the admissibility of the {\em sample average} decision rule $\delta_{SA}$ defined as  \begin{equation}\label{eq:SADR-gen}\delta_{SA}(\xi^1, \ldots, \xi^n) \in \arg\min\bigg\{ \frac{1}{n}\sum_{i=1}^n F(x, \xi^i) : x\in X\bigg\}\end{equation} 
This is a standard procedure in stochastic optimization, and often goes by the name of {\em sample average approximation (SAA)}; in machine learning, it goes by the name of {\em Empirical Risk Minimization (ERM)}. To emphasize the dependence on the number of samples $n$, we introduce a superscript, i.e., $\delta^n_{SA}$ will denote the estimator based on the sample average of the objective from $n$ observations. Moreover, for any $n \in \N$, let $\Delta^n$ be the set of all decision rules $\delta: \underbrace{\R^m \times \R^m \times \ldots \times \R^m}_{n \textrm{ times}} \to X$ such that $\E_{\xi^1, \ldots, \xi^n}[\delta(\xi^1, \ldots, \xi^n)]$ exists.


\medskip

\paragraph{Stein's paradox.} In turns out that there are simple instances of problem~\eqref{eq:original-so} where the sample average rule is {\em inadmissible}. Consider the setting of Example 2 in the Introduction, where $m=d$, $F(x, \xi) = \|x - \xi\|^2$ and $X=\R^d$. We assume $\xi$ is distributed normally $\xi \sim N(\mu,I)$ with unknown mean $\mu$; here $I$ denotes the identity matrix. In the language of statistical decision theory, the states of nature are now parametrized by $\mu \in \R^d$. Since $\E_{\xi \sim D}[F(x, \xi)] =\|x - \E[\xi]\|^2 + \mathbb{V}[\xi]$, the minimizer is simply $x(D) = \E[\xi] = \mu$ with objective value $\mathbb{V}[\xi]$. Evaluating~\eqref{eq:loss-so}, $\cL(\mu, x) = \E_{\xi \sim D}[F(x, \xi)] - \E_{\xi \sim D}[F(x(D), \xi)] = \left(\|x - \E[\xi]\|^2 + \mathbb{V}[\xi]\right) - \mathbb{V}[\xi] = \|x - \mu\|^2.$ Thus, $x = \mu$ minimizes the loss when the state of nature is $\mu$. Consequently, the problem becomes the classical problem of estimating the mean of a Gaussian from samples under ``squared distance loss". Also, the sample average decision rule solves~\eqref{eq:SADR-gen} which is the problem $\min\{ \frac{1}{n}\sum_{i=1}^n \norm{x-\xi^i}^2 : x\in \R^d\}$ and therefore returns the empirical average of the samples, i.e., $\delta^n_{SA}(\xi^1, \ldots, \xi^n) = \overline\xi$ where $\overline \xi := \frac{1}{n}\sum_{i=1}^n \xi^i$. It is well-known that this sample average decision rule is {\em inadmissible} if $d\geq 3$; this was first observed by Stein~\cite{stein1956inadmissibility} and is commonly referred to as {\em Stein's paradox} in statistics literature. The {\em James-Stein estimator}~\cite{james1961estimation} strictly dominates the sample average estimator; see~\cite{berger2013statistical,lehmann2006theory} for an exposition.
\medskip

\paragraph{Our results.} We focus on two particular cases of the stochastic optimization problem~\eqref{eq:original-so}:

\begin{enumerate}
\item $m=d$, $F(x,\xi) =\xi^Tx$, $X\subseteq \R^d$ is a given compact (not necessarily convex) set, and $\xi$ has a {\em Gaussian distribution} with {\em unknown} mean $\mu\in \R^d$ and covariance $\Sigma \in \R^{d\times d}$ denoted by $\xi \sim N(\mu, \Sigma)$. In other words, we optimize an uncertain linear objective over a fixed compact set. Note that, along with linear or convex optimization, we also capture non-convex feasible regions like mixed-integer non-linear optimization or linear complementarity constraints. 
\item $m=d$, $F(x,\xi) = \frac12\|x\|^2 - \xi^Tx$, $X\subseteq \R^d$ is a box constrained set, i.e., $X:= \{x \in \R^d: \ell_i \leq x_i \leq u_i,\;\; i = 1, \ldots, d\}$ ($\ell_i \leq u_i$ are arbitrary real numbers), and $\xi$ has a {\em Gaussian distribution} with {\em unknown} mean $\mu\in \R^d$ and covariance $\Sigma \in \R^{d\times d}$ denoted by $\xi \sim N(\mu, \Sigma)$. Here, we wish to minimize a convex quadratic function with an uncertain linear term over box constraints.
\end{enumerate}

In the first case, we show that there is no ``Stein's paradox" type phenomenon, i.e., the sample average solution is admissible for every $d \in \N$. For the second case, we show that the sample average solution is admissible for $d \leq 4$. Note that in the second situation above, $F(x, \xi)= \frac12 \|x - \xi \|^2 - \frac12 \|\xi\|^2$ and thus the problem of minimizing $\E[\frac12 \|x - \xi \|^2 + \frac12 \|\xi\|^2] = \E[\frac12\|x - \xi\|^2] + \frac12\mathbb{E}[\|\xi\|^2]$ is equivalent to the setting of Stein's paradox (since $\frac12\mathbb{E}[\|\xi\|^2]$ is just a constant), except that we now impose box constraints on $x$. Thus, admissibility is recovered for $d=3,4$ with box constraints. While we are unable to establish it for $d \geq 5$, we strongly suspect that there is no Stein's paradox in {\em any dimension} once box constraints are imposed; in fact, we believe this is true when {\em any} compact constraint set is imposed (see discussion below). 
The precise statements of our results follow.

\begin{theorem}\label{thm:linear} Consider problem~\eqref{eq:original-so} in the setting where $X$ is a given compact set and $F(\xi, x) = \xi^Tx$, and $\xi \sim N(\mu, \Sigma)$ with unknown $\mu$ and $\Sigma$. The sample average rule now simply becomes 
 \begin{equation}\label{eq:SADR}\delta^n_{SA}(\xi^1, \ldots, \xi^n) \in \arg\min\{ \overline \xi^T x : x\in X\}\end{equation} where $\overline \xi := \frac{1}{n}\sum_{i=1}^n \xi^i$ denotes the sample average of the observed objective vectors. 
For any $n \in \N$, and any $\Sigma \in \R^{d\times d}$, we consider the states of nature to be parametrized by $\mu \in \R^d$. Then for every $n \in \N$ and $\Sigma \in \R^{d\times d}$, $\delta^n_{SA}$ is admissible within $\Delta^n$.

\end{theorem}

\begin{theorem}\label{thm:quadratic} Let $d\leq 4$. Consider problem~\eqref{eq:original-so} in the setting where $X:= \{x \in \R^d: \ell_i \leq x_i \leq u_i,\;\; i = 1, \ldots, d\}$ ($\ell_i \leq u_i$ are arbitrary real numbers) and $F(\xi, x) = \frac12\|x\|^2 - \xi^Tx$, and $\xi \sim N(\mu, \Sigma)$ with unknown $\mu$ and $\Sigma$. The sample average rule now simply becomes 
 \begin{equation}\label{eq:SADR-Q}\delta^n_{SA}(\xi^1, \ldots, \xi^n) \in \arg\min\left\{ \frac12 \|x\|^2 - \overline \xi^T x : x\in X\right\}\end{equation} where $\overline \xi := \frac{1}{n}\sum_{i=1}^n \xi^i$ denotes the sample average of the observed vectors. 
For any $n \in \N$, and any $\Sigma \in \R^{d\times d}$, we consider the states of nature to be parametrized by $\mu \in \R^d$. Then for every $n \in \N$ and $\Sigma \in \R^{d\times d}$, $\delta^n_{SA}$ is admissible within $\Delta^n$.

\end{theorem}

The following result provides some concrete basis for our belief that compact constraints imply admissibility in any dimension even in the quadratic case; alas, we are unable to establish it even for box constraints when $d \geq 5$.

\begin{theorem}\label{thm:quad-ball} When $X = \{x \in \R^d: \|x \| \leq R\}$ for some $R > 0$, $F(\xi, x) = \frac12\|x\|^2 - \xi^Tx$, and $\xi \sim N(\mu, \Sigma)$, then the sample average rule is admissible for all $d\geq 1$.
\end{theorem}

We present two different proofs of Theorem~\ref{thm:linear}. The first one, presented in Sections~\ref{sec:first-proof-linear} and~\ref{sec:gen-covariance} uses a novel proof technique for admissibility, to the best of our knowledge. The second proof, presented in Section~\ref{sec:lin-proof} uses the conventional idea of showing that the sample average estimator is the (unique) Bayes estimator under an appropriate prior. We feel that the first proof technique could be useful for future research into the question of admissibility of solution estimators for stochastic optimization. The second method using Bayes estimators is easier to generalize to the quadratic settings of Theorems~\ref{thm:quadratic} and~\ref{thm:quad-ball}, and thus forms a natural segue into their proofs presented in Section~\ref{sec:quad-proof}.

\subsection{Comparison with previous work} 

The statistical decision theory perspective on stochastic optimization presented here follows the framework of~\cite{davarnia2017estimation} and~\cite{davarniabayesian}. 
In particular, the authors of~\cite{davarnia2017estimation} consider admissibility of solution estimators in two different stochastic optimization problems: one where $X = \R^d$ and $F(x,\xi) = x^TQx + \xi^T x$ for some fixed matrix positive definite matrix $Q$ (i.e., unconstrained convex quadratic minimization), and the second one where $X$ is the unit ball and $F(x,\xi) = \xi^T x$. $\xi$ is again assumed to be distributed according to a normal distribution $N(\mu, I)$ with unknown mean $\mu$. They show that the sample average approximation is {\em not} admissible in general for the first problem, and it is admissible for the second problem. Note that the second problem is a special case of our setting. In both these cases, there is a closed-form solution to the deterministic version of the optimization problem, which helps in the analysis. This is not true for the general optimization problem we consider here (even if we restrict $X$ to be a polytope, we get a linear program which, in general, has no closed form solution).

Another difference between our work and \cite{davarnia2017estimation} is the following. In \cite{davarnia2017estimation}, the question of admissibility is addressed within a smaller subset of decision rules that are ``decomposable" in the sense that any decision rule is of the form $\tau\circ \kappa$, where $\kappa: \underbrace{\R^d \times \R^d \times \ldots \times \R^d}_{n \textrm{ times}} \to \R^d$ maps the data $\xi^1, \ldots,\xi^n$ to a vector $\hat\mu \in \R^d$ and then $\tau :\R^d \to X$ is of the form $\tau(\hat\mu) \in \arg\min\{\hat\mu^T x : x \in X\}$. In other words, one first estimates the mean of the uncertain objective (using any appropriate decision rule) and then uses this estimate to solve a deterministic optimization problem. In the follow-up work~\cite{davarniabayesian}, the authors call such decision rules {\em Separate estimation-optimization (Separate EO) schemes} and more general decision rules as {\em Joint estimation-optimization (Joint EO) schemes}. In this paper, we establish admissibility of the sample average estimator within general decision rules (joint EO schemes in the terminology of~\cite{davarniabayesian}). The only condition we put on the decision rules is that of integrability, which is a minimum requirement needed to even define the risk of a decision rule. Note that proving inadmissibility within separate EO schemes implies inadmissibility within joint EO schemes. On the other hand, establishing admissibility within joint EO schemes means defending against a larger class of decision rules. The general concept of joint estimation-optimization schemes also appears in~\cite{chu2008solving,liyanage2005practical,elmachtoub2017smart}, presented in slightly different vocabulary.

As mentioned before, the quadratic convex objective has been studied in statistics in the large body of work surrounding Stein's paradox, albeit not in the stochastic optimization language that we focus on here. Moreover, all of this classical work is for the unconstrained problem. To the best of our knowledge, the version with box constraints has not been studied before (but see~\cite{charras1991bayes,casella1981estimating,marchand2001improving,karimnezhad2011estimating,kumar2008estimating,fourdrinier2010bayes,hartigan2004uniform,bickel1981minimax,gatsonis1987estimation,gorman1990lower} and the book~\cite{fourdrinier2018shrinkage} for a related, but different, statistical problem that has received a lot of attention). 
It is very intriguing (at least to us) that in the presence of such constraints, admissibility is recovered for dimensions $d = 3, 4$; recall that for the unconstrained problem, the sample average solution is admissible only for $d = 1,2$ and inadmissible for $d \geq 3$. 

\subsection{Admissibility and other notions of optimality} We end our discussion of the results with a few comments about other optimality notions for decision rules. In large sample statistics, one often considers the behavior of decision rules when $n\to \infty$ (recall $n$ is the number of samples). A sequence of decision rules $\delta^n$ (each $\delta^n$ is based on $n$ i.i.d samples) is said to be {\em asymptotically inadmissible} if there exists a decision rule sequence $\bar\delta^n$ such that $\lim_{n\to \infty} \frac{R(\theta,\bar\delta^n)}{R(\theta,\delta_n)} \leq 1$ for every $\theta$, and for some $\hat\theta$ the limiting ratio is strictly less than 1. Admissibility for every $n\in \N$ does not necessarily imply asymptotic admissibility~\cite[Problem 4, page 200]{bickel2015mathematical}, and asymptotic admissibility does not imply admissibility for finite $n\in \N$, i.e., there can be decision rules that are inadmissible for every $n\in \N$ and yet be asymptotically admissible. Thus, the small-sample behaviour (fixed $n\in \N$) and large sample behavior ($n \to \infty$) can be quite different. One advantage of proving asymptotic admissibility is that it also implies the {\em rate} of convergence of the risk (as a function of $n$) is optimal; such rules are called {\em rate optimal}. Unfortunately, our admissibility results about $\delta^n_{SA}$ do not immediately imply asymptotic admissibility or rate optimality. Standard techniques for proving rate optimality such as Hajek-Le Cam theory~\cite[Chapter 8]{van2000asymptotic} cannot be applied because regularity assumptions about the decision rules are not satisfied in our setting. For example, in the linear objective case discussed above, $\delta^n_{SA}$ has a degenerate distribution that is supported on the boundary of the feasible region $X$. Similarly, in the quadratic case, $\delta^n_{SA}$'s distribution is supported on the compact domain $X$, with majority of the mass on the boundary when $\mu$ is outside $X$. This rules out any possibiity of ``asymptotic normality" or ``local asymptotic minimaxity" results~\cite[Chapters 7, 8]{van2000asymptotic}.

While we are unable to prove rate optimality for $\delta^n_{SA}$, it is reasonably straightforward to show that $\delta^n_{SA}$ is {\em consistent} in the sense that $R(\theta, \delta^n_{SA}) \to 0$ as $n\to \infty$. This can be derived from consistency results in stochastic optimization literature~\cite[Chapter 5]{shapiro2009lectures}, but we present the argument here for completeness. In the linear objective case, the loss for $\delta^n_{SA}$ is given by $\mathcal{L}(\mu, \delta^n_{SA}) = \mu^T\delta^n_{SA} - \mu^Tx(\mu) = (\mu - \bar\xi)^T\delta^n_{SA} + \bar \xi^T\delta^n_{SA} - \mu^Tx(\mu) \leq (\mu - \bar\xi)^T\delta^n_{SA} + \bar \xi^Tx(\mu) - \mu^Tx(\mu)$ since $\delta^n_{SA}$ is the minimizer with respect to $\bar\xi$. Thus, $\mathcal{L}(\mu,\delta^n_{SA}) \leq (\mu - \bar\xi)^T(\delta^n_{SA} - x(\mu)) \leq \|\mu - \bar\xi)\|\cdot K$ by the Cauchy-Schwarz inequality, where $K$ is the diameter of the compact feasible region $X$. Therefore, $R(\mu, \delta^n_{SA})\leq K\E [\|\mu - \bar\xi)\|]$. Since the sample average $\bar\xi$ has a normal distribution with mean $\mu$ and variance that scales like $O(1/n)$, $R(\mu, \delta^n_{SA}) \to 0$ with rate $O(1/\sqrt{n})$. A similar argument can be made in the quadratic objective case. However, we are unable to show that $O(1/\sqrt{n})$ is the optimal rate in either case.

There is a large body of literature on {\em shrinkage estimators} in the unconstrained, quadratic objective setting. A relatively recent insight~\cite{xie2012sure} shows that as $d\to\infty$ (recall $d$ is the dimension), a certain class of shrinkage estimators (called {\em SURE estimators}) have risk functions that dominate any other shrinkage estimator's risk, and hence the sample average estimator's risk, with just a single sample. This potentially suggests that the phenomenon presented here, where admissibility of $\delta^n_{SA}$ is recovered for $d=3,4$, holds only for small dimensions and for large enough dimensions, the sample average estimator remains inadmissible. However, this is not immediate because of two reasons: 1) the value of $d$ for which the SURE estimator in~\cite{xie2012sure} starts to dominate any other estimator depends on the parameter $\mu$, and 2) the setting in~\cite{xie2012sure} is still unconstrained optimization. In fact, as stated earlier, we strongly suspect that with compact constraints, admissibility of $\delta^n_{SA}$ holds for {\em all} dimensions (see Theorem~\ref{thm:quad-ball}). We remark that SURE estimators need not be admissible themselves~\cite{johnstone1988inadmissibility}.

There are other notions of optimality of decision rules even in the small/finite sample setting. For example, the {\em minimax} decision rule minimizes the sup norm of the risk function, i.e., one solves $\inf_{\delta}\sup_{\theta}R(\theta,\delta)$. In general, admissibility does not imply minimaxity, nor does minimaxity imply admissibility. Of course, if a minimax rule is inadmissible, then the dominating rule is also minimax and is certainly to be preferred, unless computational concerns prohibit this. In many settings however (e.g., estimation in certain exponential and group families~\cite[Chapter 5]{lehmann2006theory}), minimax rules are also provably admissible and thus minimaxity is a more desirable criterion. 

Generally speaking, admissibility is considered a weak notion of optimality because admissible rules may have undesirable properties like very high risk values for certain states of the world. Moreover, as noted above, inadmissible rules may have optimal large sample behavior. Nevertheless, it is useful to know if widely used decision rules such as sample average approximations satisfy the basic admissibility criterion, because if not, then one could use the dominating decision rule unless it is computationally much more expensive.

\section{Technical Tools}\label{sec:tech-tools}

We first recall a basic fact from calculus.

\begin{lemma}\label{lem:calculus} Let $F: \R^m \to \R$ be a twice continuously differentiable map such that $F(0) = 0$. Suppose $\nabla^2F(0)$ is not negative semidefinite; in other words, there is a direction $d \in \R^d$ of positive curvature, i.e., $d^T\nabla^2F(0)d > 0$. Then there exists $z \in \R^m$ such that $F(z) > 0$.

\end{lemma}
\begin{proof} If $\nabla F(0) \neq 0$, then there exist $\lambda > 0$ such that $F(z) > 0$ for $z = \lambda \nabla F(0)$ since $F(0) = 0$. Else, if $\nabla F(0) = 0$ then there exists $\lambda > 0$ such that $F(z) > 0$ for $z = \lambda d$, where $d$ is the direction of positive curvature at $0$.\end{proof}

We will need the following central definition and result from statistics. See e.g., Section 6, Chapter 1 in~\cite{lehmann2006theory}.

\begin{definition} A {\em statistic} is a function $T : \chi \to \R^m$, i.e., it is any function that maps the data to a vector (or a scalar if $m=1$). Let $\cP$ be a family of distributions on the sample space $\chi$. A {\em sufficient statistic for $\cP$} is a statistic on $\chi$ such that the conditional distribution on $\chi$ given $T = t$ does not depend on the distribution from $\cP$, for all $t \in \R^m$.
\end{definition}

\begin{prop}\label{prop:avg-normal-suff} Let $\chi = \underbrace{\R^d \times \R^d \times \ldots \times \R^d}_{n \textrm{ times}}$ and let $\cP = \{ \underbrace{N(\mu,I) \times N(\mu,I) \times \ldots \times N(\mu,I)}_{n \textrm{ times}}: \mu \in \R^d \}$, i.e., $(\xi^1, \ldots, \xi^n) \in \chi$ are i.i.d samples from the normal distribution $N(\mu,I)$. Then $T(\xi^1\ldots, \xi^n) = \overline\xi :=  \frac{1}{n}\sum_{i=1}^n \xi^i$ is a sufficient statistic for $\cP$.
\end{prop}

We will also need the following useful property for the family of normal distributions $\{N(\mu, I): \mu \in \R^m\}$. Indeed the following result is true for any {\em exponential family} of distributions; see Theorem 5.8, Chapter 1 in~\cite{lehmann2006theory} for details.

\begin{theorem}\label{thm:diff-normal} Let $f: \R^m \to \R^d$ be any integrable function. The function $$h(\mu) :=
\int_{\R^m} f(y)e^{\frac{-n\| y - \mu \|^2}{2}} dy$$ is continuous and has derivatives of all orders with respect to $\mu$, which can be obtained by differentiating under the integral sign.
\end{theorem}

In the rest of the paper, for any vector $v$, $v_j$ will denote the $j$-th coordinate, and for any matrix $A \in \R^{p \times q}$, $A_{ij}$ will denote entry in the $i$-th row and $j$-th column.

We need one further result on the geometry of the hypercube which is easy to verify. We recall that for any closed, convex set $C\subseteq \R^d$ and point $x\in C$, the {\em normal cone at $x$} is defined to be the set of all vectors $c \in \R^d$ such that $x \in \argmax_{y \in C} c^T y$. We extend this concept to any face $F$ of $C$: The {\em normal cone at the face $F$} is defined to be the set of all vectors $c \in \R^d$ such that $F \subseteq \argmax_{y \in C} c^T y$.

\begin{lemma}\label{lem:hypercube-faces}
Let $X= \{x \in \R^d: -u_i \leq x_i \leq u_i \;\; i = 1, \ldots, d\}$ be a box centered at the origin. For any face $F \subseteq X$ of $X$ (possibly with $F=X$), let $I_F^+\subseteq \{1, \ldots, d\}$ be the subset of coordinates which are set to the bound $u_i$ for all points in $F$, $I_F^-\subseteq \{1, \ldots, d\}$ be the subset of coordinates which are set to the bound $-u_i$ for all points in $F$, and $N_F$ denote the normal cone at $F$. Then the following are true:
\begin{enumerate}
\item For any face $F \subseteq X$, $$F + N_F = \left\{x \in \R^d: \begin{array}{ccl} x_i \geq u_i & & i \in I_F^+ \\x_i \leq -u_i && i \in I_F^- \\
-u_i \leq x_i \leq u_i && i \not\in I_F^+ \cup I_F^- \\
\end{array}\right\}.$$
\item The interior of $F + N_F$ is disjoint from the interior of $F' + N_F'$ whenever $F \neq F'$ and we have the following decomposition of $\R^d$: $$\R^d = \bigcup_{F \textrm{ face of } X} F + N_F$$
\end{enumerate}


\end{lemma}

\section{Proof of Theorem~\ref{thm:linear} (the scenario with linear objective)}\label{sec:first-proof-linear}

\subsection{When the covariance matrix is the identity}

\begin{proof}[Proof of Theorem~\ref{thm:linear} when $\Sigma = I$] As introduced in the previous sections, $\overline\xi$ will denote the sample average of $\xi^1, \ldots, \xi^n$. Consider an arbitrary decision rule $\delta \in \Delta^n$. Consider the conditional expectation $$\eta(y) = \E_{\xi^1, \ldots, \xi^n}[\delta(\xi^1,\ldots, \xi^n)| \overline\xi = y].$$ Observe that $\eta(y) \in \conv(X)$ (i.e., the convex hull of $X$, which is compact since $X$ is compact) since $\delta$ maps into $X$. Moreover, since $\overline\xi$ is a sufficient statistic for the family of normal distributions by Proposition~\ref{prop:avg-normal-suff}, $\eta(y)$ does not depend on $\mu$. This is going to be important below. To maintain intuitive notation, we will also say that $\delta^n_{SA}$ is given by $\delta^n_{SA}(\xi^1, \ldots, \xi^n) = \eta^*(\overline \xi)$, where $\eta^*(y)$ returns a point in $\arg\min\{\;y^T x \; : \; x \in X\}$. Note also that for any action $x\in X$,~\eqref{eq:loss-so} evaluates to $$\cL(\mu, x) =  \mu^Tx - \mu^Tx(\mu),$$ where $x(\mu)$ denotes the optimal solution to the problem $\min\left\{\mu^T x : x \in X\right\}$. Using the law of total expectation, 
\[\begin{array}{rcl}
R(\mu, \delta) &= &\E_{\xi^1, \ldots, \xi^n}[\cL(\mu, \delta(\xi^1,\ldots, \xi^n))]\\
&= &\E_{\xi^1, \ldots, \xi^n}[\mu^T\delta(\xi^1,\ldots, \xi^n)] - \mu^Tx(\mu) \\
&=& \E_y[\E_{\xi^1, \ldots, \xi^n}[\mu^T\delta(\xi^1,\ldots, \xi^n)| \overline\xi = y]] - \mu^Tx(\mu) \\
&=& \E_y[\mu^T \eta(y)] - \mu^Tx(\mu)
\end{array}
\]

If $\eta = \eta^*$ almost everywhere, then $R(\mu,\delta) = R(\mu, \delta^n_{SA})$ for all $\mu \in \R^d$, and we would be done. So in the following, we assume that $\eta \neq \eta^*$ on a set of strictly positive measure. This implies the following

\begin{claim}\label{claim:diff-decision} For all $y\in \R^d$, $y^T\eta(y) \geq y^T\eta^*(y)$ and the set $\{y \in \R^d\; :\;y^T\eta(y) > y^T\eta^*(y)\}$ is of strictly positive measure.\end{claim}

\begin{proof} Since $X$ is compact, $\conv(X)$ is a compact, convex set and $\min\{ y^T x \; : \; x \in \conv(X)\} = \min\{ y^T x \; : \; x \in X\}$ for every $y\in \R^d$. Therefore, since $\eta(y) \in \conv(X)$ and $\eta^*(y) \in \arg\min\{\;y^T x \; : \; x \in X\}$, we have $y^T\eta(y) \geq y^T\eta^*(y)$ for all $y \in \R^d$.

Since $\conv(X)$ is a compact, convex set, the set of $y \in \R^d$ such that $|\arg\min\{ y^T x \; : \; x \in \conv(X)\}| > 1$ is of zero Lebesgue measure. Let $S \subseteq \R^d$ be the set of $y\in \R^d$ such that $\arg\min\{ y^T x \; : \; x \in \conv(X)\}$ is a singleton, i.e., there is a unique optimal solution; so $\R^d\setminus S$ has zero Lebesgue measure. Let $D := \{y \in \R^d: \eta(y) \neq \eta^*(y)\}$. Since we assume that $D$ has strictly positive measure, $D \cap S$ must have strictly positive measure. Consider any $y \in D\cap S$. Since $\min\{ y^T x \; : \; x \in X\} = \min\{ y^T x \; : \; x \in \conv(X)\}$, we must have $\arg\min\{ y^T x \; : \; x \in X\} \subseteq \arg\min\{ y^T x \; : \; x \in \conv(X)\}$. Since $y \in S$, $\arg\min\{ y^T x \; : \; x \in \conv(X)\}$ is a singleton and thus $\eta^*(y)$ is the unique optimum for $\min\{ y^T x \; : \; x \in \conv(X)\}$. Since $y\in D$, $\eta(y) \neq \eta(y^*)$, and therefore $y^T\eta(y) > y^T\eta^*(y)$. Thus, we have the second part of the claim.\end{proof}

Now consider the function $F: \R^d \to \R$ defined by \begin{equation}\label{eq:F}F(\mu) := R(\mu, \delta) - R(\mu, \delta^n_{SA}).\end{equation} To show that $\delta^n_{SA}$ is admissible, it suffices to show that there exists $\bar\mu \in \R^d$ such that $F(\bar \mu) > 0$. For any $\mu\in \R^d$, we have from above
$$\begin{array}{rcl}F(\mu) &= &R(\mu, \delta) - R(\mu, \delta^n_{SA}) \\ & = &\E_y[\mu^T \eta(y)] - \E_y[\mu^T \eta^*(y)]  \\
& = & \mu^T\bigintsss_{\R^d} (\frac{n}{2\pi})^{n/2}\eta(y)e^{\frac{-n\| y - \mu \|^2}{2}}dy - \mu^T\bigintsss_{\R^d} (\frac{n}{2\pi})^{n/2}\eta^*(y)e^{\frac{-n\| y - \mu \|^2}{2}}dy \\
& = & \mu^T\bigintsss_{\R^d} (\frac{n}{2\pi})^{n/2}(\eta(y) - \eta^*(y))e^{\frac{-n\| y - \mu \|^2}{2}}dy
\end{array}$$
where in the second to last equality, we have used the fact that $\overline\xi$ has distribution $N(\mu,\frac{1}{n}I)$. Note that the formula above immediately gives $F(0) = 0$. We will employ Lemma~\ref{lem:calculus} on $F(\mu)$ to show the existence of $\bar\mu \in \R^d$ such that $F(\bar\mu) > 0$. For this purpose, we need to compute the gradient $\nabla F(\mu)$ and Hessian $\nabla^2F(\mu)$. We alert the reader that in these calculations, it is crucial that $\eta(y)$ does not depend on $\mu$ (due to sufficiency of the sample average) and hence it is to be considered as a constant when computing the derivatives below. For ease of calculation, we introduce the following functions $E, G^1, \ldots, G^d: \R^d \to \R^d$:

$$
\begin{array}{rcl}
E(\mu) &:= &(\frac{n}{2\pi})^{n/2}\bigintssss_{\R^d} (\eta(y) - \eta^*(y))e^{\frac{-n\| y - \mu \|^2}{2}}dy, \\
G^i(\mu) & := & (\frac{n}{2\pi})^{n/2}\bigintssss_{\R^d} y_i(\eta(y) - \eta^*(y))e^{\frac{-n\| y - \mu \|^2}{2}}dy, \\
\qquad i & = & 1, \ldots, d.
\end{array}
$$
So $F(\mu) = \mu^TE(\mu)$. We also define the map $G :\R^d \to \R^{d\times d}$ as $$G(\mu)_{ij} = (G^i(\mu))_j.$$ 

\begin{claim}\label{claim:grad} For any $\mu\in \R^d$, $\nabla F(\mu) = E(\mu) + nG(\mu)\mu - n(\mu^TE(\mu)) \mu$. (Note that $G(\mu)\mu$ is a matrix-vector product.)
\end{claim}
\begin{cpf} This is a straightforward calculation. Consider the $i$-th coordinate of $\nabla F(\mu)$, i.e., the $i$-th partial derivative 

$$\begin{array}{rcl} \frac{\partial F}{\partial \mu_i} &= &\frac{\partial (\sum_{j}\mu_j E(\mu)_j)}{\partial \mu_i} \\
& = & E(\mu)_i + \sum_{j=1}^d\mu_j \frac{\partial E(\mu)_j}{\partial \mu_i} \\
& = & E(\mu)_i + \sum_{j=1}^d\mu_j \big(\bigintsss_{\R^d}(\frac{n}{2\pi})^{n/2}(\eta(y) - \eta^*(y))_j\frac{\partial(e^{\frac{-n\| y - \mu \|^2}{2}})}{\partial \mu_i}dy\big) \\
& = & E(\mu)_i + \mu^T\bigintsss_{\R^d}(\frac{n}{2\pi})^{n/2}(\eta(y) - \eta^*(y))\frac{\partial(e^{\frac{-n\| y - \mu \|^2}{2}})}{\partial \mu_i}dy\\
\vspace{5pt}& = & E(\mu)_i + \mu^T\bigintsss_{\R^d}(\frac{n}{2\pi})^{n/2}(\eta(y) - \eta^*(y))e^{\frac{-n\| y - \mu \|^2}{2}}(n(y_i - \mu_i))dy \\
& = & E(\mu)_i + n\mu^TG^i(\mu) - n(\mu^TE(\mu))\mu_i
 \end{array} $$
where in the third equality, we have used Theorem~\ref{thm:diff-normal} and the fact that $\eta(y), \eta^*(y)$ do not depend on $\mu$ by sufficiency of the sample average (Proposition~\ref{prop:avg-normal-suff}). The last expression above corresponds to the $i$-th coordinate of $E(\mu) + nG(\mu)\mu - n(\mu^TE(\mu)) \mu$. Thus, we are done.
\end{cpf}

\begin{claim}\label{claim:hessian} $\nabla^2 F(0) = n(G(0)^T + G(0))$.
\end{claim}
\begin{cpf} Let us compute $\frac{\partial^2 F}{\partial \mu_i\mu_j}$ using the expression for $\frac{\partial F}{\partial \mu_i}$ from Claim~\ref{claim:grad}.

$$\begin{array}{rcl} \vspace{5pt}\frac{\partial^2 F}{\partial \mu_i\mu_j} &= & \frac{\partial (E(\mu)_i)}{\partial \mu_j} +  n\frac{\partial (\mu^TG^i(\mu))}{\partial \mu_j} -  n\frac{\partial ((\mu^TE(\mu))\mu_i)}{\partial \mu_j} \\
& = & \frac{\partial (E(\mu)_i)}{\partial \mu_j} + n(G^i(\mu))_j + n\mu^T\frac{\partial (G^i)}{\partial \mu_j} - n\frac{\partial (\mu^TE(\mu))}{\partial \mu_j} \mu_i - n(\mu^TE(\mu))\gamma_{ij}
 \end{array} $$
 where $\gamma_{ij}$ denotes the Kronecker delta function, i.e., $\gamma_{ij} = 1$ if $i=j$ and $0$ otherwise. At $\mu = 0$, the above simplifies to \begin{equation}\label{eq:simpler}\frac{\partial^2 F}{\partial \mu_i\mu_j}\bigg|_{\mu=0} = \frac{\partial (E(\mu)_i)}{\partial \mu_j}\bigg|_{\mu=0} + n(G^i(0))_j.\end{equation}
Let us now investigate $\frac{\partial (E(\mu)_i)}{\partial \mu_j}$. By applying Theorem~\ref{thm:diff-normal} and the sufficiency of $\overline\xi$ again, we obtain
$$\begin{array}{rcl} \frac{\partial (E(\mu)_i)}{\partial \mu_j} & = & \bigintsss_{\R^d} (\frac{n}{2\pi})^{n/2}(\eta(y) - \eta^*(y))_i\frac{\partial(e^{\frac{-n\| y - \mu \|^2}{2}})}{\partial \mu_i}dy \\
& = & \bigintsss_{\R^d} (\frac{n}{2\pi})^{n/2}(\eta(y) - \eta^*(y))_i e^{\frac{-n\| y - \mu \|^2}{2}}(n(y_j - \mu_j))dy \\
& = &n (\frac{n}{2\pi})^{n/2}\bigintsss_{\R^d} y_j(\eta(y) - \eta^*(y))_i e^{\frac{-n\| y - \mu \|^2}{2}}dy \\
\vspace{5pt}& & - n(\frac{n}{2\pi})^{n/2}\mu_j \bigintsss_{\R^d} (\eta(y) - \eta^*(y))_i e^{\frac{-n\| y - \mu \|^2}{2}}dy \\
& = & n(G^j(\mu))_i - n\mu_jE(\mu)_i
\end{array}
$$


Therefore, at $\mu = 0$, we obtain that $\frac{\partial (E(\mu)_i)}{\partial \mu_j}\bigg|_{\mu=0} = nG^j(0)_i.$ Putting this back into~\eqref{eq:simpler}, and using the definition of the matrix $G(\mu)$, we obtain 

$$\frac{\partial^2 F}{\partial \mu_i\mu_j}\bigg|_{\mu=0} =nG^j(0)_i + n(G^i(0))_j = n(G(0)_{ji} + G(0)_{ij}).$$
Thus, we obtain that $\nabla^2 F(0) = n(G(0)^T + G(0))$. \end{cpf}
\begin{claim}\label{claim:trace} There exists a direction of positive curvature for $\nabla^2 F(0)$, i.e., there exists $d \in \R^d$ such that $d^T\nabla^2F(0)d > 0$.
\end{claim}
\begin{cpf} Consider the trace $\Tr(\nabla^2 F(0))$ of the Hessian at $\mu=0$. By Claim~\ref{claim:hessian},

$$\begin{array}{rcl}\vspace{5pt}\Tr(\nabla^2 F(0)) & = &{2n}\Tr(G(0))  \\
\vspace{5pt}& = & 2n\sum_{i=1}^d (\frac{n}{2\pi})^{n/2}\bigintsss_{\R^d} y_i(\eta(y) - \eta^*(y))_i e^{\frac{-n\| y - \mu \|^2}{2}}dy \\
& = & 2n (\frac{n}{2\pi})^{n/2}\bigintsss_{\R^d} y^T(\eta(y) - \eta^*(y)) e^{\frac{-n\| y - \mu \|^2}{2}}dy \\
\end{array}$$

By Claim~\ref{claim:diff-decision}, $y^T(\eta(y) - \eta^*(y)) \geq 0$ for any $y \in \R^d$ and $y^T(\eta(y) - \eta^*(y)) > 0$ on a set of strictly positive measure. Therefore, $\int_{\R^d} y^T(\eta(y) - \eta^*(y)) e^{\frac{-n\| y \|^2}{2}}dy > 0$.

Therefore, the trace of $\nabla^2 F(0)$ is strictly positive. Since the trace equals the sum of the eigenvalues of $\nabla^2 F(0)$ (see Section 1.2.5 in~\cite{horn-johnson}), we must have at least one strictly positive eigenvalue. The corresponding eigenvector is a direction of positive curvature.
\end{cpf}

As noted earlier, $F(0) = 0$. Combining Claim~\ref{claim:trace} and Lemma~\ref{lem:calculus}, there exists $\bar \mu \in \R^d$ such that $F(\bar\mu) > 0$.
\end{proof}

\subsection{General covariance}\label{sec:gen-covariance}

The proof in the previous section focused on the family of normal distributions with the identity as the covariance matrix. We now consider any positive definite covariance matrix $\Sigma$ for the normal distribution of $\xi$. In this case, we again consider the function $F(\mu)$ defined in~\eqref{eq:F} and prove that there exists $\bar\mu$ such that $F(\bar \mu) > 0$. The only difference is that in the formulas one must substitute the distribution $\overline\xi \sim N(\mu, \frac{1}{n}\Sigma)$, i.e., the density function everywhere must be $$g_{\mu,\Sigma}(y):= \frac{1}{\sqrt\sigma}\bigg(\frac{n}{2\pi}\bigg)^{n/2}\exp\bigg(-\frac{n}{2}(y-\mu)^T\Sigma^{-1}(y-\mu)\bigg),$$ where $\sigma$ is the determinant of $\Sigma$. 
Redefining
$$
\begin{array}{rcl}
E(\mu) &:= &\frac{1}{\sqrt\sigma}\bigg(\frac{n}{2\pi}\bigg)^{n/2}\bigintss_{\R^d} (\eta(y) - \eta^*(y))\exp\bigg(-\frac{n}{2}(y-\mu)^T\Sigma^{-1}(y-\mu)\bigg)dy, \\
G^i(\mu) & := & \frac{1}{\sqrt\sigma}\bigg(\frac{n}{2\pi}\bigg)^{n/2}\bigintss_{\R^d} y_i(\eta(y) - \eta^*(y))\exp\bigg(-\frac{n}{2}(y-\mu)^T\Sigma^{-1}(y-\mu)\bigg)dy,\; i = 1, \ldots, d,\\
\end{array}
$$
%
%
%
letting $G(\mu)$ be the matrix with $G^i(\mu)$ as rows, and adapting the calculations from the previous section reveals that
\begin{equation}\label{eq:new-hessian}\nabla^2 F(0) = n(\Sigma^{-1}G(0) + G(0)^T\Sigma^{-1})\end{equation}
Claim~\ref{claim:diff-decision} again shows that the trace 
{$\Tr(G(0)) = \bigintsss_{\R^d} y^T(\eta(y) - \eta^*(y))g_{0,\Sigma}(y)dy > 0.$}
This shows that $G(0)$ has an eigenvalue $\lambda$ with positive {\em real} part (since $G(0)$ is not guaranteed to be symmetric, its eigenvalues and eigenvectors may be complex). Let the corresponding (possibly complex) eigenvector be $v$, i.e., $G(0)v = \lambda v$ and $Re(\lambda) > 0$ (denoting the real part of $\lambda$). Following standard linear algebra notation, for any matrix/vector $M$, $M^*$ will denote its Hermitian conjugate~\cite{horn-johnson} (which equals the transpose if the matrix has real entries). We now consider 
$$\begin{array}{rcl} v^* \nabla^2 F(0) v & = & n (v^*(\Sigma^{-1}G(0) + G(0)^T\Sigma^{-1})v)\\
& = & n (v^*\Sigma^{-1}G(0)v + v^*G(0)^T\Sigma^{-1}v) \\
& = & n (v^*\Sigma^{-1}G(0)v + v^*G(0)^*\Sigma^{-1}v)  \\
& = & n (\lambda(v^*\Sigma^{-1}v) + \lambda^* (v^*\Sigma^{-1}v))  \\
& = & 2n (v^*\Sigma^{-1}v) Re(\lambda)
\end{array}$$
Since $\Sigma$ is positive definite, so is $\Sigma^{-1}$. Therefore $v^*\Sigma^{-1}v > 0$ and we obtain that $v^* \nabla^2 F(0) v > 0$. Since $\nabla^2 F(0)$ is a symmetric matrix, all its eigenvalues are real and in particular its largest eigenvalue $\gamma_d$ is positive because $$\gamma_d = \max_{x \in \C^d\setminus\{0\}} \frac{x^* \nabla^2 F(0) x}{x^* x} \geq \frac{v^* \nabla^2 F(0) v}{v^*v} > 0.$$ Thus, $\nabla^2 F(0)$ has a direction of positive curvature and  Lemma~\ref{lem:calculus} implies that there exists $\bar \mu \in \R^d$ such that $F(\bar\mu) > 0$.

\section{An alternate proof for the linear objective based on Bayes' decision rules}\label{sec:lin-proof}

To the best of our knowledge, our proof technique for admissibility from the previous sections is new. 
The conventional way of addressing admissibility uses Bayesian analysis. We recall the basic ideas behind these techniques and provide an alternate proof for Theorem~\ref{thm:linear} using these ideas, which arguably gives a simpler proof. On the other hand, this alternate proof builds upon some well-established facts in statistics, and so is less of a ``first principles" proof compared to the one presented in the previous sections. Moreover, as we noted earlier, the new technique of the previous proof might be useful for future admissibility investigations in stochastic optimization. 

We now briefly review the relevant ideas from Bayesian analysis. 
Consider a general statistical decision problem with $\Theta$ denoting the states of nature, $\cA$ denoting the set of actions, and $\{P_\theta: \theta \in \Theta\}$ the family of distributions on the sample space $\chi$. Let $P^\star$ be any so-called {\em prior distribution} on $\Theta$. For any decision rule $\delta$, one can compute the expected risk, a.k.a., the {\em Bayes' risk} $$r(P^\star, \delta):= \E_{\theta \sim P^\star}[R(\theta, \delta)].$$ A decision rule that minimizes $r(P^\star, \delta)$ is said to be a {\em Bayes' decision rule}.

\begin{theorem}\label{thm:bayes}\cite[Chapter 5, Theorem 2.4]{lehmann2006theory} If a decision rule is the unique\footnote{Here uniqueness is to be interpreted up to differences on a set of measure zero.} Bayes' decision rule for some prior, then it is admissible.
\end{theorem}

The following is the well-known statement that Gaussian distributions are {\em self-conjugate}~\cite[Example 2.2]{lehmann2006theory}.

\begin{theorem}\label{thm:conjugate} Let $d \in \N$ and let $\Sigma \in \R^{d\times d}$ be fixed. For the joint distribution on $(\xi, \mu) \in (\underbrace{\R^d \times \R^d \times \ldots \times \R^d}_{n \textrm{times}}) \times \R^d$ defined by $\xi | \mu \sim \underbrace{\mathcal{N}(\mu, \Sigma)\times \mathcal{N}(\mu, \Sigma) \times \ldots \times \mathcal{N}(\mu, \Sigma)}_{n \textrm{times}}$ and $\mu \sim \mathcal{N}(\mu_0, \Sigma_0)$, we have that $\mu | \xi \sim \mathcal{N}((\Sigma_0^{-1} + n\Sigma^{-1})^{-1}(\Sigma_0^{-1}\mu_0 + n \Sigma^{-1}\overline\xi), (\Sigma_0^{-1} + n\Sigma^{-1})^{-1}).$ \end{theorem}
\bigskip

\begin{proof}[Alternate proof of Theorem~\ref{thm:linear}]
Consider the prior $P^\star$ to be $\mu \sim \mathcal{N}(0, \Sigma)$, then by Theorem~\ref{thm:conjugate}, $\mu | \xi \sim \mathcal{N}\big(\frac{n}{n+1}\overline \xi, \frac{1}{n+1}\Sigma \big).$ In particular, the mean of $\mu$, conditioned on the observation $\xi$ is simply a scaling of the sample average $\overline \xi$. Now we do a standard Bayesian analysis:

$$\begin{array}{rcl}r(P^\star, \delta) &= &\E_{\mu \sim \mathcal{N}(0, \Sigma)}[\E_{\xi \sim \mathcal{N}(\mu, \Sigma)^n}[\cL(\mu, \delta)] ] \\
& = & \bigintss_{\mu} \bigintss_{\xi} (\mu^T\delta(\xi) - \mu^Tx(\mu)) p(\xi|\mu)p(\mu)d\xi d\mu \\
& = & \bigintss_{\mu} \bigintss_{\xi} \mu^T\delta(\xi) p(\xi|\mu)p(\mu)d\xi d\mu - C
\end{array}$$
where $x(\mu)$ again denotes the optimal solution to $\min\left\{\mu^T x : x \in X\right\}$, $p(\xi|\mu)$ denotes the conditional density function of $\xi|\mu$, $p(\mu)$ is the density function of the prior on $\mu$, and the constant $C$ equals $\bigintss_{\mu} \bigintss_{\xi} \mu^Tx(\mu) p(\xi|\mu)p(\mu)d\xi d\mu = \bigintss \mu^Tx(\mu) p(\mu)d\mu$. To find the decision rule $\delta$ that minimizes $r(P^\star,\delta)$, we thus need to minimize $\bigintss_{\mu} \bigintss_{\xi} \mu^T\delta(\xi) p(\xi|\mu)p(\mu)d\xi d\mu$. We change the order of integration by Fubini's theorem, and rewrite $$r(P^\star, \delta) = \bigintss_{\xi} \bigintss_{\mu} \mu^T\delta(\xi) p(\mu|\xi)p(\xi)d\mu d\xi.$$ Consequently, given the observation $\xi$, we choose $\delta(\xi)\in X$ that minimizes the inner integral $$\bigintss_{\mu} \mu^T\delta(\xi) p(\mu|\xi)d\mu = \frac{n}{n+1}\overline\xi^T \delta(\xi).$$ Thus, we may set $\delta(\xi)$ to be the minimizer in $X$ for the linear objective vector $\frac{n}{n+1}\overline\xi^T$, which is just a scaling of the sample average. Except for a set of measure zero, any linear objective has a unique solution as was noted in the proof of Claim~\ref{claim:diff-decision}. Thus, the Bayes' decision rule is unique and coincides with the sample average estimator $\delta^n_{SA}$. We are done by appealing to Theorem~\ref{thm:bayes}.
\end{proof}
%
%
%

\section{Proof of Theorem~\ref{thm:quadratic} (the scenario with quadratic objective)}\label{sec:quad-proof}

%
%
A more general version of Theorem~\ref{thm:bayes} goes by the name of {\em Blyth's method}. 
Here, we state it as in~\cite{lehmann2006theory} (see Exercise 7.12 in Chapter 5).

\begin{theorem}\label{thm:blyth-thm}
Let $\Theta \subseteq \R^d$ be any open set of states of nature. Suppose $\delta$ is a decision rule with a continuous risk function and $\{\pi_n\}_{n \in \N}$ is a sequence of prior distributions such that:
\begin{enumerate}
	\item $r(\pi_n, \delta) = \int \cR(\theta, \delta)d\pi_n  < \infty$ for all $n \in \N$, where $r$ is the Bayes risk.
	\item For any nonempty open subset $\Theta_0 \subseteq \Theta,$ we have
	$$\lim_{n \to \infty} \frac{r(\pi_n, \delta) - r(\pi_n, \delta^{\pi_n})}{\int_{\Theta_0}\pi_n(\theta)d\theta}= 0
	$$ where $\delta^{\pi_n}$ is a Bayes decision rule having finite Bayes risk with respect to the prior density $\pi_n.$
Then, $\delta$ is an admissible decision rule. 
\end{enumerate}
\end{theorem} 


\begin{proof}[Proof of Theorem~\ref{thm:quadratic}] For simplicity of exposition, we consider $X$ to be centered at 0, i.e., $\ell_i = -u_i$. The entire proof can be reproduced for the general case by translating the means of the priors to the center of axis-aligned box $X$. The calculations are much easier to read and follow when we assume the origin to be the center. We will show that $\delta^n_{SA}$ satisfies the conditions of Theorem~\ref{thm:blyth-thm} under the priors $$\mu \sim \cN(0, \tau^2 \Sigma), \;\; \tau \in \N, \text{and } \Sigma \text{ is the covariance matrix of } \xi | \mu.$$ First we obtain a simple expression for the loss function $\cL(\mu, x)$ for any action $x \in X$ under the state of nature $\mu\in \R^d$. As noted in Section~\ref{sec:results}, $\E_\xi[F(x,\xi)] = \E_\xi[\frac12 x^Tx - \xi^Tx] = \frac12 x^Tx - \mu^Tx = \frac12\|\mu - x\|^2 - \frac12\|\mu\|^2$. Let the minimum value of this for $x\in X$ be denoted by $B(\mu)$. Thus, $$\begin{array}{rcl}\cL(\mu, x) &= &\E_{\xi \sim \cN(\mu, \Sigma)}[F(x, \xi)] - \E_{\xi \sim D}[F(x(D), \xi)] \\ & = & \frac12\|\mu - x\|^2 - \frac12\|\mu\|^2 - B(\mu)\end{array}$$
Putting this into the numerator of the second condition in Theorem~\ref{thm:blyth-thm} and simplifying (which means that the terms $- \frac12\|\mu\|^2 - B(\mu)$ cancel out), we need to show that for any open set $\Omega_0\subseteq \R^d$, 
 

\begin{equation}\label{eq:blyth-formula}\lim_{\tau \to \infty}\frac{\int_{\R^d}\int_{\R^d} \norm{ \mu - \delta^n_{SA}(\xi) }^2 p(\xi | \mu) \pi_{\tau}(\mu) d\xi d\mu - \int_{\R^d}\int_{\R^d} \norm{ \mu - \delta_{Bayes}(\xi)}^2 p(\xi | \mu) \pi_{\tau}(\mu) d\xi d\mu } {\int_{\Omega_0} \pi_{\tau}(\mu)d\mu} = 0 \end{equation}
where $p(\xi|\mu) = \cN(\mu, \Sigma)$ denotes the conditional density of $\xi$ given $\mu$, and $\pi_\tau(\mu)$ is the marginal density of $\mu$ (of course, the marginal density $\pi_\tau(\mu)$ is nothing but the prior $\cN(0, \tau^2\Sigma)).$
\medskip

Next, let us see how the rule $\delta^n_{SA}$ given by~\eqref{eq:SADR-Q} behaves. Minimizing $\frac12 \|x\|^2 - \overline \xi^T x$ is equivalent to minimizing $\frac12 \|x - \overline\xi\|^2$ since $\overline \xi$ can be regarded as constant for the optimization problem $\min_{x \in X} \frac12 \|x\|^2 - \overline \xi^T x$. Thus, $\delta^n_{SA}$ returns the closest point to $\overline \xi$ in $X$, i.e., \begin{equation}\label{eq:SA-proj}\delta^n_{SA}(\xi) = \Proj_X(\overline \xi)\end{equation}
where the notation $\Proj_X(y)$ denotes the projection of the closes point in $X$ to $y$. 
\medskip

Let us also see what the Bayes' rule $\delta_{Bayes}$ is, i.e., what value of $\delta_{Bayes}(\xi)$ minimizes $$\begin{array}{rcl}\int_{\R^d}\int_{\R^d} \norm{ \mu - \delta^n_{SA}(\xi) }^2 p(\xi | \mu) \pi_{\tau}(\mu) d\xi d\mu  & = & \int_{\R^d}\int_{\R^d} \norm{ \mu - \delta_{Bayes}(\xi)}^2 p(\mu | \xi) m(\xi) d\mu d\xi \\
& = & \int_{\R^d} \E_{\mu|\xi} \norm{ \mu - \delta_{Bayes}(\xi) }^2 m(\xi) d\xi.\end{array}$$ where we have again evaluated the Bayes' risks by switching the order of the integrals and using the conditional density \begin{equation}\label{eq:conditional}p(\mu | \xi) = \cN\left(\frac{n\tau^2}{n\tau^2 + 1}\overline{\xi}, \frac{\tau^2}{n\tau^2 + 1}\Sigma\right) \qquad (\text{by Theorem}~\ref{thm:conjugate}),\end{equation} and the marginal density of $\xi$ is denoted by $m(\xi)$. Since \begin{equation}\label{eq:exp-norm} \E[\norm{ Y - a}^2] = \norm{\E[Y] -a}^2 + \mathbb{V}[Y] \end{equation}
for any random variable $Y$ and constant $a\in \R$, one sees that \begin{equation}\label{eq:Bayes-proj}\delta_{Bayes}(\xi) = \Proj_X\left(\frac{n\tau^2}{n\tau^2 + 1}\overline{\xi}\right).\end{equation}
%
%
%


\bigskip

\noindent Let us now consider the numerator and denominator of the left hand side in \eqref{eq:blyth-formula} separately.

\paragraph{Numerator:}

The numerator is: 
$$\begin{array}{rl}&\bigintsss_{\R^d}\bigintsss_{\R^d} \norm{ \mu - \delta^n_{SA}(\xi) }^2 p(\xi | \mu) \pi_{\tau}(\mu)  d\xi d\mu - \bigintsss_{\R^d}\bigintsss_{\R^d} \norm{ \mu - \delta_{Bayes}(\xi)}^2 p(\xi | \mu) \pi_{\tau}(\mu) d\xi d\mu \\
\\
 =& \bigintsss_{\R^d}\bigintsss_{\R^d} \norm{ \mu - \Proj_X (\overline\xi) }^2 p(\xi | \mu) \pi_{\tau}(\mu) d\xi d\mu - \bigintsss_{\R^d}\bigintsss_{\R^d} \norm{ \mu - \Proj_X (\frac{n\tau^2}{n\tau^2 + 1}\overline\xi)}^2 p(\xi | \mu) \pi_{\tau}(\mu)  d\xi d\mu\\
 \\
 = & \bigintsss_{\R^d}\bigintsss_{\R^d}\bigintsss_{\R^d} \norm{ \mu - \Proj_X (\overline\xi) }^2 p(\xi | \overline\xi) p(\overline\xi | \mu) \pi_{\tau}(\mu)   d\xi d\overline{\xi} d\mu\\
 & -  \bigintsss_{\R^d}\bigintsss_{\R^d}\bigintsss_{\R^d} \norm{ \mu - \Proj_X (\frac{n\tau^2}{n\tau^2 + 1}\overline\xi)}^2 p(\xi | \overline\xi) p(\overline\xi | \mu) \pi_{\tau}(\mu) d\xi d\overline{\xi} d\mu \\
 \\
= & \bigintsss_{\R^d}\bigintsss_{\R^d} \norm{ \mu - \Proj_X (\overline\xi) }^2 p(\overline{\xi} | \mu) \pi_{\tau}(\mu) d\overline{\xi} d\mu - \bigintsss_{\R^d}\bigintsss_{\R^d} \norm{ \mu - \Proj_X (\frac{n\tau^2}{n\tau^2 + 1}\overline\xi)}^2 p(\overline{\xi} | \mu) \pi_{\tau}(\mu) d\overline{\xi} d\mu \\
\\
 = & \bigintsss_{\R^d}\bigintsss_{\R^d} \norm{ \mu - \Proj_X (\overline\xi) }^2 p(\mu | \overline{\xi}) m(\overline{\xi}) d\mu d\overline{\xi} - \bigintsss_{\R^d}\bigintsss_{\R^d} \norm{ \mu - \Proj_X (\frac{n\tau^2}{n\tau^2 + 1}\overline\xi)}^2 p(\mu | \overline{\xi}) m(\overline{\xi}) d\mu d\overline{\xi}, \end{array}$$
where the first equality follows from substituting ~\eqref{eq:SA-proj} and~\eqref{eq:Bayes-proj}, and the last equality follows from the standard trick in Bayesian analysis of switching the order of the integrals.

Since $\overline{\xi} | \mu \sim \mathcal{N}(\mu, \frac{1}{n}\Sigma)$, and $\pi(\mu) \sim \mathcal{N}(0, \tau^2 \Sigma)$, it is a simple exercise to check that $\mu | \overline{\xi} \sim \mathcal{N}(\frac{n\tau^2}{n\tau^2 + 1}\overline{\xi}, \frac{\tau^2}{n\tau^2 + 1}\Sigma),$ and $m(\overline{\xi}) \sim \mathcal{N}(0, \frac{n\tau^2 + 1}{n}\Sigma). $ 
The formula for the numerator above can then be rewritten as \begin{equation}\label{eq:numerator-blyth}\begin{array}{rl}&\bigintsss_{\R^d} \E_{\mu|\overline{\xi}} \norm{ \mu -  \Proj_X (\overline{\xi}) }^2 m(\overline{\xi}) d\overline{\xi} - \bigintsss_{\R^d} \E_{\mu|\overline{\xi}} \norm{ \mu - \Proj_X (\frac{n\tau^2}{n\tau^2 + 1}\overline{\xi})}^2 m(\overline{\xi}) d\overline{\xi} \\
\\
%
%
%
=&\bigintsss_{\R^d} \mathbb{V}(\mu|\overline{\xi}) + \norm{\E_{\mu | \overline{\xi}} \mu - \Proj_X (\overline{\xi})}^2 m(\overline{\xi}) d\overline{\xi} - \bigintsss_{R^d} \mathbb{V}(\mu|\overline{\xi}) + \norm{\E_{\mu | \overline{\xi}} \mu - \Proj_X (\frac{n\tau^2}{n\tau^2 + 1}\overline{\xi})}^2 m(\overline{\xi}) d\overline{\xi} \\
\\
=& \bigintsss_{\R^d} \norm{\frac{n\tau^2}{n\tau^2 + 1}\overline{\xi} -  \Proj_X (\overline\xi)}^2 - \norm{\frac{n\tau^2}{n\tau^2 + 1}\overline{\xi} - \Proj_X (\frac{n\tau^2}{n\tau^2 + 1}\overline{\xi})}^2 m(\overline{\xi}) d\overline{\xi}.\\
\\
 = &\bigintsss_{\R^d} \left( \norm{\Proj_X (\frac{n\tau^2}{n\tau^2 + 1}\overline{\xi}) -  \Proj_X (\overline{\xi})}^2 
 - 2\left\langle{\frac{n\tau^2}{n\tau^2 + 1}\overline{\xi} - \Proj_X (\frac{n\tau^2}{n\tau^2 + 1}\overline{\xi}), \Proj_X (\overline{\xi}) - \Proj_X (\frac{n\tau^2}{n\tau^2 + 1}\overline{\xi}})\right\rangle\right)m(\overline{\xi}) d\overline{\xi},  \end{array}\end{equation}
where the first equality follows from~\eqref{eq:exp-norm}, the second equality follows from the formula for the conditional density $p(\mu|\overline{\xi})$, and the last equality follows from the fact that 
$$
\norm{ a - b }^2 - \norm{ a - c }^2 = \norm{ c - b }^2 - 2 \langle{a - c,b - c}\rangle, \quad \textrm{for any three vectors }a,b,c \in \R^d.
$$

\begin{claim}\label{claim:hypercube-facet}
Consider the box $X= \{x \in \R^d: -u_i \leq x_i \leq u_i \;\; i = 1, \ldots, d\}$. 
Let $F$ be any face of $X$ and so $F' := \frac{n\tau^2 + 1}{n\tau^2}F$ is a face of $X' := \frac{n\tau^2 + 1}{n\tau^2}X$. Suppose $\overline\xi \in F' + N_{F'}$ where $N_{F'}$ denotes the normal cone at $F'$ with respect to $X'$ (see Lemma~\ref{lem:hypercube-faces} and the discussion above it), then $\Proj_X \big(\frac{n\tau^2}{n\tau^2 + 1}\overline\xi\big)$ and $\Proj_X (\overline\xi)$ both lie in $F$.  Consequently, \begin{equation}\label{eq:orthogonal-face}\left\langle{\frac{n\tau^2}{n\tau^2 + 1}\overline\xi - \Proj_X \big(\frac{n\tau^2}{n\tau^2 + 1}\overline\xi\big), \Proj_X (\overline\xi) - \Proj_X \big(\frac{n\tau^2}{n\tau^2 + 1}\overline\xi\big)}\right\rangle = 0.\end{equation}
\end{claim}
\begin{proof}
%

\begin{figure}[h]
  \includegraphics[width=\linewidth]{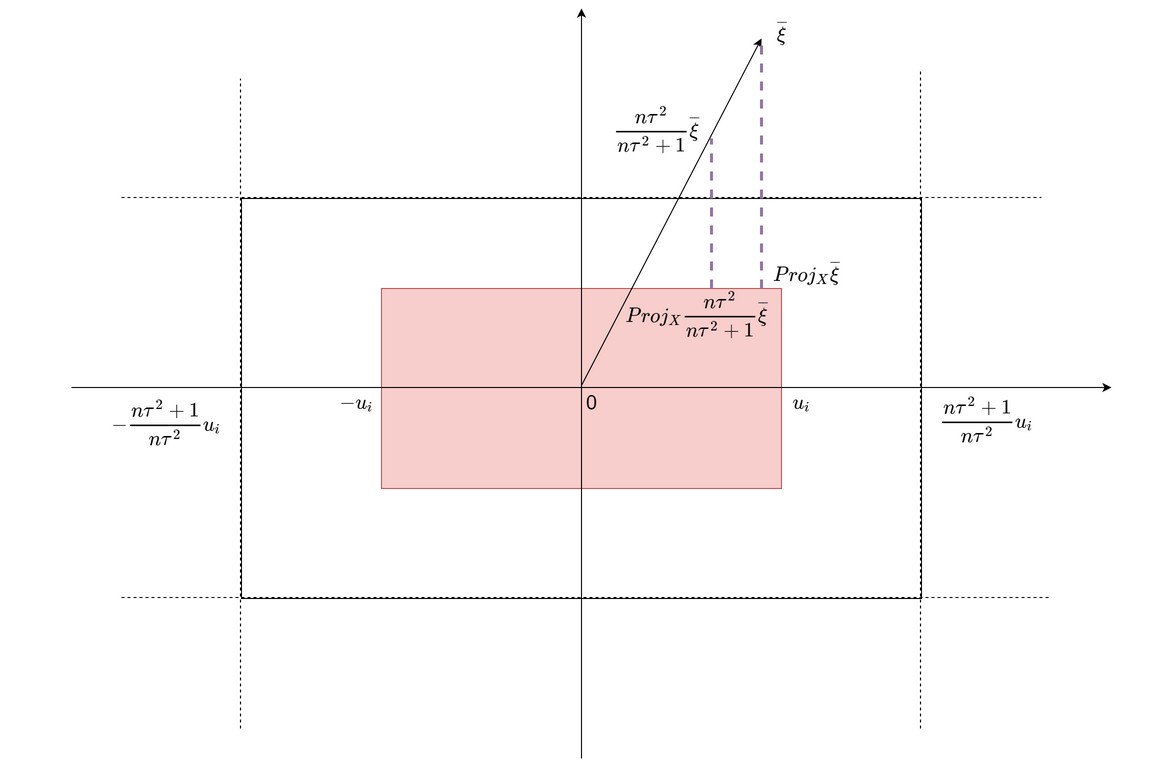}
  \caption{The scaling $X'$ of $X$ and the different regions $F' + N_{F'}$ for different faces $F$ of $X$, with an illustration of Claim~\ref{claim:hypercube-facet}.}\label{fig:hypercube}
\end{figure}

The general formula for the projection onto $X$ is given by \begin{equation}\label{eq:proj-formula}\Proj_X(y) = p, \;\; \textrm{ where } \;\; p_i = sign(y_i)min\{|y_i|, u_i\}.\end{equation} Consider any face $F$ of $X$ and any $\overline\xi \in F' + N_{F'}$; see Figure~\ref{fig:hypercube}. By Lemma~\ref{lem:hypercube-faces} part 1., $|\overline \xi_i| \geq \frac{n\tau^2 + 1}{n\tau^2} |u_i|$ for all $i \in I_F^+ \cup I_F^-$. Therefore, $|\overline\xi_i| \geq \frac{n\tau^2}{n\tau^2 + 1}|\overline\xi_i| \geq |u_i|$ for all $i \in I_F^+ \cup I_F^-$. By the projection formula~\eqref{eq:proj-formula}, for all $i \in I_F^+ \cup I_F^-$ the $i$-th coordinate of both $\Proj_X \big(\frac{n\tau^2}{n\tau^2 + 1}\overline\xi\big)$ and $\Proj_X (\overline\xi)$ are equal to $sign(\overline\xi_i)u_i$. This shows that they both lie on $F$. By the geometry of projections, the vector $\frac{n\tau^2}{n\tau^2 + 1}\overline\xi - \Proj_X \big(\frac{n\tau^2}{n\tau^2 + 1}\overline\xi\big)$ is orthogonal to the face $F$ that contains the projection $\Proj_X \big(\frac{n\tau^2}{n\tau^2 + 1}\overline\xi\big)$. This proves~\eqref{eq:orthogonal-face}.
\end{proof}
%
By appealing to~\eqref{eq:orthogonal-face}, one can reduce~\eqref{eq:numerator-blyth} to \begin{equation}\label{num-simple}\bigintsss_{\R^d} \norm{\Proj_X \big(\frac{n\tau^2}{n\tau^2 + 1}\overline{\xi}\big) -  \Proj_X (\overline{\xi})}^2 m(\overline{\xi}) d\overline{\xi}\end{equation}
Using Lemma~\ref{lem:hypercube-faces} part 2., we decompose the above integral as follows:

$$\eqref{num-simple} = \sum_{F \textrm{ face of } X} \int_{F' + N_{F'}} \norm{\Proj_X \big(\frac{n\tau^2}{n\tau^2 + 1}\overline{\xi}\big) -  \Proj_X (\overline{\xi})}^2 m(\overline{\xi}) d\overline{\xi}.$$
where we have used the notation of Claim~\ref{claim:hypercube-facet} for $F'$. Using the formula in Lemma~\ref{lem:hypercube-faces} part 1., we may simplify the integral further by introducing some additional notation. For any face $F$ of $X$, recall the notation $I_F^+$ and $I_F^-$ from Lemma~\ref{lem:hypercube-faces}. Let $I_F^0 := \{1, \ldots, d\} \setminus (I_F^+ \cup I_F^-)$. We introduce the decomposition of any vector $y \in \R^d$ into $y^+, y^-, y^0$ where $y^+ \in \R^{I_F^+}$ denotes the restriction of the vector $y$ onto the coordinates in $I_F^+$; similarly, $y^- \in \R^{I_F^-}$ is $y$ restricted to $I_F^-$, and $y^0$ is $y$ restricted to $I_F^0$. Denote the corresponding domains $D_F^+ : = \{z \in \R^{I_F^+}: z_i \geq \frac{n\tau^2 + 1}{n\tau^2}u_i\}$, $D_F^- : = \{z \in \R^{I_F^i}: z_i \leq -\frac{n\tau^2 + 1}{n\tau^2}u_i\}$ and $D_F^0 : = \{z \in \R^{I_F^0}: -\frac{n\tau^2 + 1}{n\tau^2}u_i \leq z_i \leq \frac{n\tau^2 + 1}{n\tau^2}u_i\}$. By Lemma~\ref{lem:hypercube-faces} part 1.,

$$\begin{array}{rcl}\eqref{num-simple} & = &
\mathlarger{\sum_{F \textrm{ face of } X}} \bigints_{D_F^0}\bigints_{D_F^-}\bigints_{D_F^+}  \norm{\Proj_X (\frac{n\tau^2}{n\tau^2 + 1}\overline\xi) -  \Proj_X (\overline\xi)}^2 m(\overline\xi) d\overline\xi^+d\overline\xi^- d\overline\xi^0 \\
\\
& \leq & \mathlarger{\sum_{F \textrm{ face of } X}} \bigints_{D_F^0}\bigints_{D_F^-}\bigints_{D_F^+} \;\;\bigg(\frac{1}{(n\tau^2 + 1)^2}\sum_{i\in I_F^0}\overline \xi_i^2\bigg) \;\; m(\overline\xi)d\overline\xi^+d\overline\xi^- d\overline\xi^0 \\
\end{array}$$ where the inequality follows from Claim~\ref{claim:hypercube-facet} which tells us that both $\Proj_X \big(\frac{n\tau^2}{n\tau^2 + 1}\overline{\xi}\big)$ and $\Proj_X (\overline{\xi})$ lie on $F$ and therefore coincide on the coordinates in $I_F^+ \cup I_F^-$; thus, the coordinates in $I_F^+ \cup I_F^-$ vanish in the integrand. Moreover, on any remaining coordinate $i$ that is not set to the bound $u_i$ or $-u_i$, the absolute difference in the coordinate is at most $|\frac{n\tau^2}{n\tau^2 + 1}\overline\xi_i -  \overline\xi_i| = \frac{1}{n\tau^2 + 1}|\overline\xi_i|$, by the projection formula~\eqref{eq:proj-formula}. 
Plugging in the formula for $m(\overline\xi)$, we get

$$
\begin{array}{rcl}
\eqref{num-simple} & \leq & \left(\frac{n}{2\pi(n\tau^2 + 1)\det(\Sigma)^{1/d}}\right)^{d/2}\mathlarger{\sum_{F \textrm{ face of } X}} \bigints_{D_F^0}\bigints_{D_F^-}\bigints_{D_F^+} \;\;\bigg(\frac{1}{(n\tau^2 + 1)^2}\sum_{i\in I_F^0}\overline \xi_i^2\bigg) \;\;  e^\frac{-n\overline\xi^T\Sigma^{-1}\overline\xi}{2(n\tau^2 + 1)}d\overline\xi^+d\overline\xi^- d\overline\xi^0 \\
& = & \left(\frac{n}{2\pi(n\tau^2 + 1)\det(\Sigma)^{1/d}}\right)^{d/2}\mathlarger{\sum_{F \textrm{ face of } X}} \bigints_{D_F^0} \;\;\bigg(\frac{1}{(n\tau^2 + 1)^2}\sum_{i\in I_F^0}\overline \xi_i^2\bigg) \;\; \bigints_{D_F^+} \bigints_{D_F^-}   e^\frac{-n\overline\xi^T\Sigma^{-1}\overline\xi}{2(n\tau^2 + 1)}d\overline\xi^+d\overline\xi^- d\overline\xi^0 \\
& = & \left(\frac{n}{2\pi(n\tau^2 + 1)\det(\Sigma)^{1/d}}\right)^{d/2}\mathlarger{\sum_{F \textrm{ face of } X}} \bigints_{D_F^0} \;\;\bigg(\frac{1}{(n\tau^2 + 1)^2}\sum_{i\in I_F^0}\overline \xi_i^2\bigg) \;\; \cdot C_F (n\tau^2 + 1)^{(d-\dim(F))/2} h(\overline\xi^0,\tau)d\overline\xi^0
\end{array}
$$
where the integral is evaluated using standard Gaussian integral formulas and $C_F$ is a constant independent of $\tau$, $\dim(F)$ denotes the dimension of the face $F$ and $h(\xi^0, \tau)$ is a continuous function which is upper bounded on the compact domain $D_F^0$ for every $F$ by a universal constant independent of $\tau$. 

We now observe that if $I_F^0 = \emptyset$, i.e, $F$ is a vertex of $X$, then the integrand is simply $0$ (in other words, when $\overline\xi$ lies in the normal cone of a vertex $v$ of $X$ translated by $\frac{n\tau^2+1}{n\tau^2}v$, the projections $\Proj_X \big(\frac{n\tau^2}{n\tau^2 + 1}\overline{\xi}\big)$ and $\Proj_X (\overline{\xi})$ are both equal to $v$, and the integral vanishes). Therefore, we are left with the terms where the face has dimensional least 1. Thus, $(n\tau^2 + 1)^{(d-\dim(F))/2}$ can be upper bounded by $(n\tau^2 + 1)^{(d-1)/2}$. Since $D_F^0$ is a compact domain and $h$ is continuous function upper bounded by a universal constant independent of $\tau$, we infer the following upper bound on the numerator $$
 \left(\frac{n}{2\pi(n\tau^2 + 1)\det(\Sigma)^{1/d}}\right)^{d/2}\cdot C \cdot (n\tau^2 + 1)^{(d-5)/2}$$ where $C$ is a constant independent of $\tau$.


\paragraph{Denominator:}

Using the density formula for $\pi_{\tau}(\mu)$, the denominator of \eqref{eq:blyth-formula} is $$\frac{1}{[2\pi\tau^2]^{d/2} \det(\Sigma)^\frac{1}{2}}{\int_{\Omega_0}e^\frac{-\mu^T\Sigma^{-1}\mu}{2\tau^2}d\mu}.$$
\bigskip

\noindent Combining the formulas for the numerator and denominator in~\eqref{eq:blyth-formula}, we have:
$$\begin{array}{rl} & \lim_{\tau \to \infty}\frac{\int_{\R^d}\int_{\R^d} \norm{ \mu - \delta^n_{SA}(\xi) }^2 p(\xi | \mu) \pi_{\tau}(\mu) d\xi d\mu - \int_{\R^d}\int_{\R^d} \norm{ \mu - \delta_{Bayes}(\xi)}^2 p(\xi | \mu) \pi_{\tau}(\mu) d\xi d\mu} {\int_{\Omega_0} \pi_{\tau}(\xi)d\xi} \\
\\
\leq & \lim_{\tau \to \infty} \left(\frac{n\tau^2}{n\tau^2 + 1}\right)^{d/2} \cdot C (n\tau^2 + 1)^{(d-5)/2}\cdot \left(1/\int_{\Omega_0}e^\frac{-\mu^T\Sigma^{-1}\mu}{2\tau^2}d\mu\right)
\end{array}$$

As $\tau \to \infty$, $\int_{\Omega_0}e^\frac{-\mu^T\Sigma^{-1}\mu}{2\tau^2}d\mu$ approaches the volume of $\Omega_0$ which is strictly positive and the first term $(\frac{n\tau^2}{n\tau^2 + 1})^{d/2}$ goes to 1. Moreover, since $d \leq 4$, the middle term $(n\tau^2 + 1)^{(d-5)/2}$ goes to zero.
Consequently, we have:
$$\lim_{\tau \to \infty}\frac{\int_{\R^d}\int_{\R^d} \norm{ \mu - \delta^n_{SA}(\xi) }^2 p(\mu | \xi) m(\xi) d\mu d\xi - \int_{\R^d}\int_{\R^d} \norm{ \mu - \delta_{Bayes}}^2 p(\mu | \xi) m(\xi) d\mu d\xi} {\int_{\Omega_0} \pi_{\tau}(\mu)d\mu} = 0 $$
By Theorem~\ref{thm:blyth-thm}, $\delta^n_{SA}$ is admissible. \end{proof}


%

\begin{proof}[Proof of Theorem~\ref{thm:quad-ball}] The proof when $X$ is a ball follows the same logic as the hypercube case of Theorem~\ref{thm:quadratic}. Observe that now the Bayes' rule $\delta_{Bayes}$ differs from $\delta_{SA}$ only on the compact domain $\frac{n\tau^2 + 1}{n\tau^2}X$; outside this, both give the same point on the boundary of $X$. Thus, the numerator retains a factor of $(n\tau^2+1)^{-2}$, as opposed to the factor of $(n\tau^2 + 1)^{(d-5)/2}$ in the analysis of the hypercube case. Hence, for all $d \geq 1$, the ratio goes to 0 as $\tau \to \infty$.
\end{proof}

\section{Future Work}

To the best of our knowledge, a thorough investigation of the admissibility of solution estimators for stochastic optimization problems has not been undertaken in the statistics or optimization literature. There are several avenues for continuing this line of investigation:

\begin{enumerate} \item The most immediate question is whether the sample average solution for the quadratic objective subject to box constraints, as considered in this paper, continues to be admissible for dimension $d \geq 5$. We strongly suspect this to be true, but our current proof techniques are not able to resolve this either way. If one considers the James-Stein estimators for $\mu$ and then uses these to solve the constrained optimization problems, the standard arguments for inadmissibility break down because of the presence of the box constraints. The problem seems to be quite different, and significantly more complicated, compared to the unconstrained case that has been studied in classical statistics literature. 

\item The next step, after resolving the higher dimension question, would be to consider general convex quadratic objectives $F(x, \xi) = x^TQx + \xi^Tx$ for some fixed positive (semi)definite matrix $Q$ and the constraint $X$ to be a general compact, convex set, as opposed to just box constraints. We believe new ideas beyond the techniques introduced in this paper are needed to analyze the admissibility of the sample average estimator for this convex quadratic program\footnote{When $Q$ is the identity and $X$ is a scaled and translated $\ell_2$ unit norm ball, as opposed to a box, the problem becomes equivalent to minimizing a linear function over the ball. Theorem~\ref{thm:linear} then applies to show admissibility in every dimension. This special case is also analyzed in~\cite{davarnia2017estimation}.}. This problem is interesting from a financial engineering perspective, where the stochastic optimization problem seeks to minimize a coherent risk measure over a convex set. The simplest such measure is a weighted sum of the expectation and the variance of the returns, which can be modeled using the above $F(x,\xi)$.
\item One may also choose to avoid nonlinearities and stick to piecewise linear $F(x,\xi)$ and polyhedral $X$. Such objectives show up in the stochastic optimization literature under the name of {\em news-vendor type problems}. The current techniques of this paper do not easily apply directly to this setting either. In fact, in the simplest setting for the news-vendor problem, one has a function $F :\R\ \times \R \to \R$ given by $F(x,\xi) = cx - p\min\{x, \xi\} + r\max\{0, x - \xi\}$ and $X = [0,U]$ for known constants $0 < r < c < p$ and some given bound $U >0$. In this setting, the natural distributions for $\xi$ are not normal, but distributions whose support is contained in the nonnegative real axis. As a starting point, one can consider the uniform distribution setting where the mean of the uniform distribution, or the width of the uniform distribution or both are unknown.
\item For learning problems, such as neural network training with squared or logistic loss, what can be said about the admissibility of the sample average rule, which usually goes under the name of ``empirical risk minimization"? Is the empirical risk minimization rule an admissible rule in the sense of statistical decision theory? It would be very interesting if the answer actually depends on the hypothesis class that is being learnt. It is also possible that decision rules that take the empirical risk objective and report a {\em local} optimum can be shown to dominate decision rules that report the global optimum, under certain conditions.  This would be an interesting perspective on the debate whether local solutions are ``better" in a theoretical sense than global optima.
\end{enumerate}

\paragraph{Acknowledgments.} We are extremely grateful to Prof. Daniel Naiman for helping us understand basic admissibility results from statistics and numerous follow-up discussions related to this work. He gave his time unquestioningly whenever we approached him. We also owe a great philosophical debt to Prof. Carey Priebe who always nudged us to think about optimization in statistical/probabilistic terms. His aphorism ``The solution to an optimization problem is a random variable!" made us think about optimization in a new light. Prof. Priebe also had numerous discussions about the paper and gave unwavering encouragement and support. Prof. Jim Spall and Long Wang also gave useful input on a draft of the paper. Insightful comments and important pointers to prior work from two anonymous referees helped organize the results with more context. 
Finally, the main inspiration for this work came from listening to a beautiful talk by Prof. G\'erard Cornu\'ejols explaining his work from~\cite{davarnia2017estimation} and~\cite{davarniabayesian}.


\bibliographystyle{plain}
\bibliography{../full-bib}


\end{document}